\documentclass[draft,reqno]{amsart}
\usepackage[utf8]{inputenc}
\usepackage{amssymb}
\usepackage{euscript}
\usepackage{extarrows}
\usepackage{mathrsfs}
\usepackage{a4wide}
\usepackage{setspace}

\makeatletter
\@namedef{subjclassname@2010}{%
\textup{2010} Mathematics Subject Classification}
\makeatother

\newtheorem{thm}{Theorem}
\newtheorem{cor}[thm]{Corollary}
\newtheorem{lem}[thm]{Lemma}
\newtheorem{pro}[thm]{Proposition}

\theoremstyle{remark}
\newtheorem{rem}[thm]{Remark}

\theoremstyle{definition}
\newtheorem{exa}[thm]{Example}

\DeclareMathOperator{\D}{d\hspace{-0.25ex}}

\DeclareMathOperator{\M}{m}

\newcommand*{\ascr}{\mathscr A}

\newcommand*{\borel}[1]{{\mathfrak B}(#1)}

\newcommand*{\cbb}{\mathbb C}
\newcommand*{\cfw}{C_{\phi,w}}
\newcommand*{\cfwa}{C_{\phi,w_\alpha}}

\newcommand*{\cfww}{C_{\phi,\widetilde w}}

\newcommand*{\esf}{\mathsf{E}}
\newcommand*{\efw}{\mathsf{E}_{\phi,w}}
\newcommand*{\efwa}{\mathsf{E}_{\phi,w_\alpha}}

\newcommand*{\dz}[1]{{\EuScript D}(#1)}

\newcommand*{\ee}{\EuScript E}

\newcommand*{\ff}{\mathcal F}

\newcommand*{\Ge}{\geqslant}

\newcommand*{\hh}{\mathcal H}

\newcommand*{\hsf}{{\mathsf h}}
\newcommand*{\hfw}{{\mathsf h}_{\phi,w}}
\newcommand*{\hfwa}{{\mathsf h}_{\phi,w_\alpha}}

\newcommand*{\is}[2]{\langle#1,#2\rangle}
\newcommand*{\jd}[1]{\EuScript N(#1)}

\newcommand*{\Le}{\leqslant}

\newcommand*{\nbb}{\mathbb N}

\newcommand*{\rbb}{\mathbb R}
\newcommand*{\rbop}{{\overline{\rbb}_+}}

\newcommand*{\smalloplus}{\raise0pt\hbox{$\scriptscriptstyle \oplus$}}

\newcommand*{\zbb}{\mathbb Z}

\newcommand*{\dfwa}{\Delta_\alpha(\cfw)}
\newcommand*{\dfwan}[1]{\varDelta_\alpha^{#1}(\cfw)}

\begin{document}
\setstretch{1.1}
\title[Aluthge transforms of composition operators]{Aluthge transforms of unbounded weighted composition operators in $L^2$-spaces}

\author[C. Benhida]{Chafiq Benhida}
\address{Chafiq Benhida, Laboratoire Paul Painlevé, Université Lille 1, 59655 Villeneuve d'Ascq, France}
\email{benhida@math.univ-lille1.fr}

\author[P.\ Budzy\'{n}ski]{Piotr Budzy\'{n}ski}
\address{Piotr Budzy\'{n}ski, Katedra Zastosowa\'{n} Matematyki, Uniwersytet Rolniczy w Krakowie, ul.\ Balicka 253c, 30-198 Kra\-k\'ow, Poland}
\email{piotr.budzynski@ur.krakow.pl}

\author[J. Trepkowski]{Jacek Trepkowski}
\address{Jacek Trepkowski, Instytut Matematyki, Uniwersytet Jagiello\'nski, ul.\ \L ojasiewicza 6, 30-348 Kra\-k\'ow, Poland}
\email{jacek.trepkowski@gmail.com}

\maketitle

\begin{abstract}
We describe the Aluthge transform of an unbounded weighted composition operator acting in an $L^2$-space. We show that its closure is also a weighted composition operator with the same symbol and a modified weight function. We investigate its dense definiteness. We characterize $p$-hyponormality of unbounded weighted composition operators and provide results on how it is affected by the  Aluthge transformation. We show that the only fixed points of the  Aluthge transformation on weighted composition operators are quasinormal ones. 
\end{abstract}

\section{Introduction}
The Aluthge transform of a bounded operator T was introduced by Aluthge to investigate the properties of $p$-hyponormal operators. He showed that the Aluthge transform of a $p$-hyponormal operator is itself a $(p+\frac12)$-hyponormal operator (see \cite{alu-ieot-1990, alu-ieot-1996}; see also \cite{hur-pams-1997}). Various connections between an operator and its Aluthge transform were later on studied by Jung, Ko, and Pearcy, in particular with regard to the invariant subspace problem. They proved that the Aluthge transformation preserves various spectra of operators and spectral pictures, and that an operator possesses a nontrivial (hyper)invariant subspace whenever its Aluthge transform does (see \cite{j-k-p-ieot-2000, j-k-p-ieot-2001}). They studied also iteration of the Aluthge transformation and conjectured that the sequence of consecutive iterates of the Aluthge transform of an operator is convergent to a quasinormal operator (see \cite{j-k-p-ieot-2000}). Although there are many cases in which the conjecture holds (see \cite{j-k-p-ieot-2000}), it was disproved in general by Thompson (see \cite{rio} more info on that matter). Further studies on the Aluthge transforms were carried (see e.g. \cite{b-z-sm-2009, b-j-laa-2010, c-j-l-ieot-2005, d-s-tams-2009, exn-jot-2009, f-j-k-p-pjm-2003, gar-ieot-2008, j-p-l-ieot-2007, l-l-j-ieot-2012, wu-laa-2002, yam-pams-2002}). Many results in the literature devoted to Aluthge transforms dealt with concrete operators.

Unbounded weighted composition operators (in $L^2$-spaces) and related operators, like composition operators or weighted shifts on directed trees, have recently become the subject of intensive study (see \cite{b-j-j-sW} and e.g., \cite{b-d-p-p-sm-2017, b-d-p-bjm-2017, b-j-j-s-jmaa-2013, b-j-j-s-2014-ampa, b-j-j-s-jfa-2015, b-j-j-s-aim-2017, j-j-s-mams-2012, c-p-t-ieot-2017, c-p-t-dm-2017, c-t-jlms-2016, pie-jmaa-2016}). These studies are mostly concerned with problems related to subnormality, reflexivity or analyticity. They show that the unbounded counterparts of the very classical objects of operator theory, which bounded weighted composition operators are, have very interesting and sometimes quite surprising properties. Among the many papers on the subject,  \cite{tre-2015-jmaa} focuses on the Aluthge transforms of weighted shifts on directed trees. It provides examples of weighted shifts on directed trees, and composition operators on discrete measure spaces as well, the Aluthge transforms of whom have pathological properties. In particular, it is shown that the Aluthge transform of a hyponormal weighted shift on a directed tree or a hyponormal composition operator may have trivial domain, hence cannot by hyponormal.

In this paper we take the studies of \cite{tre-2015-jmaa} a step further and investigate the Aluthge transforms of weighted composition operators. We first provide a full description of the $\alpha$-Aluthge transform $\dfwa$ of an unbounded weighted composition operator $\cfw$ and show that it is always closable and its closure is an associated weighted composition operator $\cfwa$ with the same symbol $\phi$ and a modified weight function $w_\alpha$ (see Theorem \ref{aluthge}). Next, we deliver sufficient conditions for the equality of $\dfwa$ and $\cfwa$ (see Proposition \ref{serwis}). Then, we turn to the problem of when the associated weighted composition operator $\cfwa$ is densely defined and show that dense definiteness of $k$th power of $\cfw$, with $k>1$, is sufficient (see Proposition \ref{zimno} and Corollary \ref{zimno2}). The next section of the paper deals with $p$-hyponormality of $\dfwa$. We begin by proving a characterization of $p$-hyponormality for unbounded weighted composition operators (see Theorem \ref{p-hyp}). Because of the relation between $\dfwa$ and $\cfwa$, using the above mentioned characterization of $p$-hyponormality, we are able to describe how Aluthge transformation affects $p$-hyponormality of $\cfw$ (see Theorems \ref{ptaszki} and \ref{second}). In particular, we obtain an unbounded weighted composition counterpart of the classical result of Aluthge on this subject. We finish the paper by showing that the only situation when $\dfwa$ equals $\cfw$ is when $\cfw$ is quasinormal. All our result are illustrated with various examples.

It should be mentioned, that in this paper weighted composition operators are not defined as products of multiplication operators and composition operators and we do not impose any additional assumptions on the underlying measure spaces, symbols or weight functions. Doing otherwise could be very restrictive, as discussed \cite{b-j-j-sW} in detail, however such a restrictive approach prevail in the literature on weighted composition operators and composition operators.

\section{Preliminaries}
In all what follows $\zbb$, $\rbb$, and $\cbb$ stand for the sets of integers, real numbers, and complex numbers, respectively. We denote by $\nbb$, $\zbb_+$, and $\rbb_+$ the sets of positive integers, nonnegative integers, and nonegative real numbers, respectively. Set $\rbop = \rbb_+ \cup
\{\infty\}$. The characteristic function of a subset $\varOmega$ of a set $X$ is denoted by $\chi_\varOmega$; if no confusion can arise, $\mathbf{1}$ stands for $\chi_X$.  The $\sigma$-algebra of all Borel sets of a topological space $Z$ is denoted by $\borel{Z}$. For any $\cbb$- or $\rbop$-valued function $f$ on $X$ and $\square\in\{=,>,\neq\}$, we write $\{f\,\square\,0\}$ for $\{x\in X\colon f(x)\,\square\, 0\}$, whenever the expression make sense; the same applies for $0$ replaced by $\infty$.

Let $A$ be an (linear) operator in a complex Hilbert space $\hh$. We denote by $\dz{A}$, $\jd{A}$, $\bar A$, and $A^*$ the domain, the kernel, the closure, and the adjoint of $A$ (in case they exist) respectively. We write $\|\cdot\|_A$ for the graph norm of $A$, i.e., $\|f\|_A^2 = \|f\|^2 + \|Af\|^2$ for $f\in \dz{A}$. We call a vector subspace $\ee$ of $\dz{A}$ a {\em core} for $A$ if $\ee$ is dense in $\dz{A}$ with respect to the graph norm of $A$. In case  $A$ is closable, $\ee$ is a core for $A$ if and only if $\bar A = \overline{A|_{\ee}}$. Given two operators $A$ and $B$ in $\hh$, we write $A \subseteq B$ if $\dz{A} \subseteq \dz{B}$ and $Af=Bf$ for all $f\in \dz{A}$.

Let $A$ be a closed, densely defined operator in a Hilbert space $\hh$. Then $A$ has a (unique) polar decomposition $A = U|A|$, where $U$ is a partial isometry on $\hh$ such that $\jd{A}=\jd{U}$ and $|A|$ is the square root of $A^*A$ (see \cite[Section 8.1]{bir-sol}). Let $p\in(0,\infty)$. We say that $A$ is $p$-hyponormal if $\dz{|A|^p}\subseteq \dz{|A^*|^p}$ and $\||A^*|^p f\|\Le \| |A|^p f\|$ for every $f\in \dz{|A|^p}$. Given $\alpha\in (0,1]$,
\begin{align*}
\Delta_\alpha(A) = |A|^\alpha U|A|^{1-\alpha}.
\end{align*}
denotes the $\alpha$-Aluthge transform of $A$ (see \cite{alu-ieot-1990, hur-pams-1997}).

Let $(X,\ascr, \mu)$ be a $\sigma$-finite measure space and $\phi$ be a {\em transformation} of $X$ (i.e., a mapping $\phi\colon X\to X$). Write $\phi^{-1}(\ascr) = \{\phi^{-1}(\varDelta)\colon \varDelta \in \ascr\}$. The transformation $\phi$ of $X$ is said to be {\em $\ascr$-measurable} if $\phi^{-1}(\ascr) \subseteq \ascr$. Let $\phi$ be an $\ascr$-measurable transformation. If $\nu\colon \ascr\to \rbop$ is a measure, then $\nu \circ \phi^{-1}$ denotes the measure on $\ascr$ defined by $\nu \circ \phi^{-1} (\varDelta) = \nu(\phi^{-1} (\varDelta))$ for $\varDelta \in \ascr$. Let $w$ be a complex $\ascr$-measurable function on $X$. Define the measure $\mu_w\colon \ascr \to \rbop$ by
\begin{align*}
\mu_w(\varDelta) = \int_{\varDelta} |w|^2 \D\mu \text{ for } \varDelta \in \ascr.
  \end{align*}
Clearly, $\mu_w$ is $\sigma$-finite. If $\mu_w \circ \phi^{-1} \ll \mu$, i.e., $\mu_w \circ \phi^{-1}$ is absolutely continuous with respect to $\mu$, then by the Radon-Nikodym theorem
(cf.\ \cite[Theorem 2.2.1]{ash}) there exists a unique (up to a set of $\mu$-measure zero) $\ascr$-measurable function $\hfw\colon X \to \rbop$ such that
 \begin{align*}
\mu_w \circ \phi^{-1}(\varDelta) = \int_{\varDelta} \hfw \D \mu, \quad
\varDelta \in \ascr.
\end{align*}
It follows from \cite[Theorem 1.6.12]{ash} and \cite[Theorem 1.29]{rud} that for every $\ascr$-measurable function $f \colon X \to \rbop$ (or for every $\ascr$-measurable function $f\colon X \to \cbb$ such that $f\circ \phi \in L^1(\mu_w)$),
\begin{align} \label{l2}
\int_X f \circ \phi \D\mu_w = \int_X f \, \hfw \D \mu.
\end{align}

Let $(X,\ascr,\mu)$ be a $\sigma$-finite measure space. As usual, $L^2(\mu)$ stands for the complex Hilbert space of all square $\mu$-integrable $\ascr$-measurable complex functions on $X$ (with the standard inner product). Let $w$ be an $\ascr$-measurable complex function on $X$ and $\phi$ be an $\ascr$-measurable transformation of $X$ such that $\mu_w\circ\phi^{-1}\ll\mu$.  Then a {\em weighted composition operator}
\begin{align*}
\cfw \colon L^2(\mu) \supseteq \dz{\cfw} \to L^2(\mu)
\end{align*}
given by
\begin{align*}
\dz{\cfw} & = \{f \in L^2(\mu) \colon w \cdot (f\circ \phi) \in L^2(\mu)\},
   \\
\cfw f & = w \cdot (f\circ \phi), \quad f \in \dz{\cfw}
\end{align*}
is well-defined. We call $\phi$ and $w$ the {\em symbol} and the {\em weight} of $\cfw$ respectively. It was shown in \cite[Proposition 9]{b-j-j-sW} that every well-defined composition operator is closed and its domain can be described as follows
\begin{align}\label{dziedzina}
\dz{\cfw}  = L^2((1+\hfw)\D \mu).
\end{align}
The latter implies the characterization of dense definiteness (see \cite[Proposition 10]{b-j-j-sW}):
\begin{align} \label{dense}
\text{$\cfw$ is densely defined if and only if $\hfw < \infty$ a.e.\ $[\mu]$.}
\end{align}
It is well-known that the boundedness of $\cfw$ can also be characterized in terms of $\hfw$:
\begin{align} \label{boundedness}
\text{$\cfw$ is bounded on $L^2(\mu)$ if and only if $\hfw$ is essentially bounded with respect to $\mu$.}
\end{align}
Considering weight function $w\equiv 1$ leads to a large subclass of weighted composition operators, the so-called composition operators. More precisely, assuming that $\phi\colon X\to X$ is $\ascr$-measurable and satisfies $\mu\circ\phi^{-1}\ll\mu$, the operator $C_{\phi, 1}$ is well-defined. We call it the {\em composition operator} induced by $\phi$; for simplicity we use the notation $C_\phi:=C_{\phi,1}$. The corresponding Radon-Nikodym derivative $\mathsf{h}_{\phi,1}$ is denoted by $\mathsf{h}_\phi$; similarly, we write $\esf_\phi(\cdot)$ for the conditional expectation $\esf_{\phi,1}(\cdot)$ (see the second paragraph following Theorem \ref{polar} below for information on the conditional expectation).

The following theorem gives a full description of the polar decomposition of $\cfw$. It is worth recalling that $\hfw\circ\phi>0$ a.e.\ $[\mu_w]$, hence the right-hand side of the equality in condition (ii) below is well-defined (see \cite[Lemma 6]{b-j-j-sW}).
\begin{thm}[\mbox{\cite[Proposition 18]{b-j-j-sW}}]\label{polar}
Assume that $\cfw$ is densely defined. Let $\cfw = U|\cfw|$ be the polar decomposition of $\cfw$. Then
\begin{enumerate}
\item[(i)] $\dz{|\cfw|}=L^2\big((1+\hfw)\D\mu\big)$ and $|\cfw|f = \hfw^{1/2}f$ for $f\in \dz{|\cfw|}$,
\item[(ii)] $U = \cfww$, where $\widetilde w\colon X \to \cbb$ is an $\ascr$-measurable function such that
\begin{align}\label{triple}
\widetilde w = w \cdot \frac{1}{(\hfw \circ \phi)^{1/2}} \text{ a.e.\ $[\mu]$.}
\end{align}
\end{enumerate}
\end{thm}
One can see that in view of the condition (i) of Theorem \ref{polar}, $|\cfw|$ is in fact equal to $M_{\hfw^{1/2}}$, the operator of multiplication by $\hfw^{1/2}$ in $L^2(\mu)$.

We proceed now by providing few necessary facts concerning the conditional expectation with respect to the $\sigma$-algebra $\phi^{-1}(\ascr)$. Assuming that $\hfw < \infty$ a.e.\ $[\mu]$, the measure $\mu_w|_{\phi^{-1}(\ascr)}$ is $\sigma$-finite (see \cite[Proposition 10]{b-j-j-sW}), and thus the conditional expectation
$\efw(f):=\mathsf{E}(f;\phi^{-1}(\ascr),\mu_w)$ with respect to the $\sigma$-algebra $\phi^{-1}(\ascr)$ and the measure $\mu_w$ of an $\ascr$-measurable function $f\colon X\to\rbop$ or a function $f\in L^p(\mu_w)$ with some $p\in[1,\infty]$ is well-defined. We will use frequently the following equalities
\begin{align}\label{conexp}
\int_X g \circ \phi \cdot f \D \mu_w = \int_X g \circ \phi \cdot \efw(f) \D \mu_w,
\end{align}
and
\begin{align}\label{conexp+}
\efw(g \circ \phi \cdot f) = g \circ \phi \cdot \efw(f),
\end{align}
that hold for $\ascr$-measurable functions $f,g \colon X \to \rbop$ or $f,g \colon X \to \cbb$ such that $f \in L^p(\mu_w)$ and $g\circ \phi\in L^q(\mu_w)$ with $p,q \in [1,\infty]$ satisfying $\frac{1}{p} + \frac {1}{q} =1$. It is well-known that for any $\ascr$-measurable function $f\colon X \to \rbop$ or $f\in L^p(\mu_w)$) we have $\efw(f) = g\circ \phi$ a.e.\ $[\mu_w]$ with some $\ascr$-measurable $\rbop$-valued or $\cbb$-valued function $g$ on $X$ such that $g=g\cdot \chi_{\{\hfw>0\}}$ a.e.\ $[\mu]$ (see \cite[Proposition 14]{b-j-j-sW}). Setting $\efw(f) \circ \phi^{-1} :=
g$ a.e.\ $[\mu]$ and using \cite[Lemma 5 and Proposition 14]{b-j-j-sW} we arrive at the following helpful equality
\begin{align} \label{fifi}
(\efw(f) \circ \phi^{-1})\circ \phi = \efw(f) \quad \text{a.e.\ $[\mu_w]$.}
\end{align}
In the second part of the paper we will frequently considers powers of conditional expectation of of functions. To make some formulas more compact we will stick to the notation
\begin{align*}
\efw^\alpha(f):=\Big(\efw(f)\Big)^\alpha,\quad \efw^\alpha(f)\circ\phi^{-1}:=\Big(\efw(f)\circ\phi^{-1}\Big)^\alpha,
\end{align*}
for any $\alpha\in\rbb$ and any $f$ such the above expressions make sense. For more information on the conditional expectation and its usage in the context of unbounded weighted composition operators see \cite{b-j-j-sW}.

We finish the Preliminaries with the following description of the adjoint of a weighted composition operator which comes in handy when studying $p$-hyponormality of weighted composition operators.
\begin{thm}[\mbox{\cite[Proposition 17]{b-j-j-sW}}]\label{adjoint}
Assume that $\cfw$ is densely defined. Then the following equalities hold$:$
\begin{align*}
\dz{\cfw^*}&=\big\{f\in L^2(\mu)\colon \hfw\cdot\efw(f_w)\circ\phi^{-1}\in L^2(\mu)\big\},\\
\cfw^*f&=\hfw\cdot\efw(f_w)\circ\phi^{-1},\quad f\in\dz{\cfw^*},
\end{align*}
where $f_w=\chi_{\{w\neq 0\}}\frac{f}{w}$.
\end{thm}
For more information on unbounded weighted composition operators in $L^2$-spaces the reader is referred to \cite{b-j-j-sW}.

\section{Aluthge transforms of wco's}
In this section we provide a formula for the $\alpha$-Aluthge transform of a weighted composition operators. For convenience, brevity, and future reference we single out the following set of assumptions.
\begin{align} \label{stand1}
\begin{minipage}{80ex}
$(X,\ascr,\mu)$ is a $\sigma$-finite measure space, $w$ is an $\ascr$-meas\-urable complex function on $X$ and $\phi$ is an $\ascr$-measurable transformation of $X$ such that $\mu_w \circ \phi^{-1} \ll \mu$.
   \end{minipage}
   \end{align}
We begin with a basic lemma.
\begin{lem}\label{relacje}
Assume \eqref{stand1}. Suppose that $\cfw$ is densely defined. Given $\alpha\in(0,1]$, let $w_\alpha\colon X\to\cbb$ be an $\ascr$-measurable function such that
\begin{align}\label{waga}
w_\alpha=w\cdot \bigg(\frac{\hfw}{\hfw\circ\phi}\bigg)^{\alpha/2}\quad \text{a.e.\ $[\mu]$}.
\end{align}
Then the following conditions are satisfied$:$
\begin{itemize}
\item[(i)] $\mu_{w_\alpha}\circ\phi^{-1}\ll\mu$,
\item[(ii)] $\mathsf{h}_{\phi, w_\alpha}=\efw(\hfw^\alpha)\circ \phi^{-1}\cdot \hfw^{1-\alpha}$ a.e. $[\mu]$,
\item[(iii)] if $\hfwa<\infty$ a.e. $[\mu]$, then
\begin{align*}
\efwa(f)\cdot\efw(\hfw^\alpha)=\efw(f\hfw^\alpha)\quad \text{a.e.\ $[\mu_w]$}
\end{align*}
for every $\ascr$-measurable function $f\colon X\to\rbop$.
\end{itemize}
\end{lem}
\begin{proof}
Fix $\alpha\in(0,1]$. 

Note that for any $\sigma\in\ascr$ such that $\mu_w(\sigma)=0$ we have $\mu\big(\{w\neq 0\} \cap \sigma\big)=0$ and consequently, in view of $\hfw<\infty$ a.e.\ $[\mu]$, we get $\mu\big(\{w_\alpha \neq 0\} \cap \sigma\big)=0$. This and the well-definiteness of $\cfw$ implies that for any $\sigma\in\ascr$ satisfying $\mu(\sigma)=0$ we have $\mu_{w_\alpha}\big(\phi^{-1}(\sigma)\big)=0$. This gives (i).

In view of \eqref{l2}, \eqref{conexp}, and \eqref{fifi} we have\allowdisplaybreaks
\begin{align*}
\int_\sigma \hfwa\D\mu
&=\int \chi_{\sigma} \hfwa \D\mu
=\int \chi_{\sigma}\circ \phi \D\mu_{w_\alpha}
=\int_{\phi^{-1}(\sigma)} |w_\alpha|^2\D\mu
=\int_{\phi^{-1}(\sigma)}\frac{\hfw^\alpha}{\hfw^\alpha\circ\phi}\D\mu_w\\
&=\int_{\phi^{-1}(\sigma)} \frac{\big(\efw(\hfw^\alpha)\circ\phi^{-1}\big)\circ\phi }{\hfw^\alpha\circ\phi } \D\mu_w
=\int_\sigma \frac{\efw(\hfw^\alpha)\circ\phi^{-1}}{\hfw^\alpha} \hfw\D\mu\\
&=\int_\sigma \efw(\hfw^\alpha)\circ\phi^{-1} \hfw^{1-\alpha}\D\mu,\quad \sigma\in\ascr,
\end{align*}
which yields (ii). Observe that, if $\alpha=1$, we also use the fact that $\efw(\hfw^\alpha)\circ\phi^{-1}=\efw(\hfw^\alpha)\circ\phi^{-1}\chi_{\{\hfwa>0\}}$.

Let $f\colon X\to\rbop$ be $\ascr$-measurable. Then by \eqref{conexp} we get
\begin{align*}
\int_{\phi^{-1}(\sigma)}\efw(f\hfw^\alpha)\D\mu_w
&=\int_{\phi^{-1}(\sigma)}f\hfw^\alpha\D\mu_w
=\int_{\phi^{-1}(\sigma)}(\hfw^\alpha\circ\phi)\cdot f\D\mu_{w_\alpha}\\
&=\int_{\phi^{-1}(\sigma)}(\hfw^\alpha\circ\phi)\cdot \efwa(f)\D\mu_{w_\alpha}
=\int_{\phi^{-1}(\sigma)}\efwa(f)\cdot\hfw^\alpha\D\mu_w\\
&=\int_{\phi^{-1}(\sigma)}\efwa(f)\cdot \efw(\hfw^\alpha)\D\mu_w,\quad \sigma\in\ascr.
\end{align*}
This implies (iii) and completes the proof.
\end{proof}
The theorem below generalizes the description of Alutghe transforms of weighted shifts on directed trees given in \cite[Theorem 4.1]{tre-2015-jmaa} and of bounded weighted composition operators given in \cite[Lemma 2.7]{azi-jab-bkms-2010} (see the discussion on the limitations of their approach in the Introduction).
\begin{thm}\label{aluthge}
Assume \eqref{stand1}. Let $\alpha\in(0,1]$. Suppose that $\cfw$ is densely defined. Then the following conditions are satisfied$:$
\begin{enumerate}
\item[(i)] $\dz{\Delta_\alpha(\cfw)}=L^2\Big(\big(1+\hfw^{1-\alpha}+\efw(\hfw^\alpha)\circ\phi^{-1}\hfw^{1-\alpha}\big)\D\mu\Big)$,
\item[(ii)] $\Delta_\alpha(\cfw)\subseteq \cfwa$, where $w_\alpha\colon X \to \cbb$ is given by \eqref{waga},
\item[(iii)] $\Delta_\alpha(\cfw)$ is closable and $\overline{\Delta_\alpha(\cfw)}=\cfwa$,
\item[(iv)] $\Delta_\alpha(\cfw)=\cfwa$ if and only if $\hfw^{1-\alpha}\Le c\Big(1+\efw(\hfw^\alpha)\circ\phi^{-1}\hfw^{1-\alpha}\Big)$ a.e. $[\mu]$ for a positive constant $c$,
\item[(v)] $\Delta_\alpha(\cfw)$ is closed if and only if $\hfw^{1-\alpha}\Le c\Big(1+\efw(\hfw^\alpha)\circ\phi^{-1}\hfw^{1-\alpha}\Big)$ a.e. $[\mu]$ for a positive constant $c$.
\end{enumerate}
\end{thm}
\begin{proof}
(i) By Theorem \ref{polar}, $f\in\dz{\Delta_\alpha(\cfw)}$ if and only if $f\in\dz{M_{\hfw^{(1-\alpha)/2}}}$ and $\widetilde w \cdot(\hfw\circ\phi)^{(1-\alpha)/2} f\circ \phi\in\dz{M_{\hfw^{\alpha/2}}}$. The first holds if and only if  $f\in L^2\big((1+\hfw^{1-\alpha})\D\mu\big)$. The second, by \eqref{triple}, translates into two following conditions
\begin{align}
\int_X (\hfw\circ\phi)^{-\alpha} |f\circ\phi|^2\D\mu_w<\infty,\label{v2}\\
\int_X \hfw^\alpha (\hfw\circ\phi)^{-\alpha} |f\circ\phi|^2\D\mu_w<\infty\label{v2'}.
\end{align}
By \eqref{l2}, the inequality in \eqref{v2} holds if and only if $f\in L^2\big(\hfw^{1-\alpha}\D\mu\big)$. In turn, in view of \eqref{conexp}, \eqref{fifi}, and  \eqref{l2}, the equality in \eqref{v2'} is satisfied if and only if $f\in L^2\Big(\efw\big(\hfw^\alpha\big)\circ\phi^{-1}\hfw^{1-\alpha}\D\mu\Big)$. Combining all the above we get the desired equality.

(ii) By Lemma \ref{relacje}\,(i), $\cfwa$ is well-defined. Now, by \eqref{dziedzina} and Lemma \ref{relacje}\,(ii), we see that
\begin{align*}
\dz{\cfwa}= L^2\Big(\big(1+\efw(\hfw^\alpha)\circ\phi^{-1}\hfw^{1-\alpha}\big)\D\mu\Big).
\end{align*}
This, in view of (i), implies that $\dz{\Delta_\alpha(\cfw)}\subseteq\dz{\cfwa}$. Now, by Theorem \ref{polar}, we have
\begin{align*}
\Delta_\alpha(\cfw)f= \hfw^{\frac{\alpha}{2}}\cdot \widetilde w\cdot \Big(\hfw^{\frac{1-\alpha}{2}} f\Big)\circ \phi=\cfwa f,\quad f\in \dz{\Delta_\alpha(\cfw)}.
\end{align*}
This proves (ii).

(iii) Since every weighted composition operator is closed, closability of $\Delta_\alpha(\cfw)$ follows from (ii). Proving that $\overline{\Delta_\alpha(\cfw)}=\cfwa$ amounts to showing that $\dz{\Delta_\alpha(\cfw)}$ is dense in $\dz{\cfwa}$ in the graph norm of $\cfwa$. Since, by \eqref{l2} and Lemma \ref{relacje}\,(ii), we have
\begin{align*}
\|f\|^2_{\cfwa}=\int_X |f|^2\Big(1+\efw(\hfw^\alpha)\circ\phi^{-1}\hfw^{1-\alpha}\Big)\D\mu,\quad f\in \dz{\cfwa},
\end{align*}
and $\hfw<\infty$ a.e. $[\mu]$, it suffices now to use \cite[Lemma 12.1]{b-j-j-s-2014-ampa} to get the required density.

(iv) According to (ii), $\varDelta_\alpha(\cfw)=\cfwa$ if and only if $\dz{\varDelta_\alpha(\cfw)}=\dz{\cfwa}$, which, by (i) and \eqref{dziedzina}, is satisfied if and only if
\begin{align*}
L^2\Big(\big(1+\hfw^{1-\alpha}+\efw(\hfw^\alpha)\circ\phi^{-1}\hfw^{1-\alpha}\big)\D\mu\Big)=L^2\Big((1+\hfwa)\D\mu\Big).
\end{align*}
By Lemma \ref{relacje}\,(ii), the above is equivalent to
\begin{align*}
L^2\Big(\big(1+\hfw^{1-\alpha}+\efw(\hfw^\alpha)\circ\phi^{-1}\hfw^{1-\alpha}\big)\D\mu\Big)
=L^2\Big((1+\efw(\hfw^\alpha)\circ\phi^{-1}\hfw^{1-\alpha})\D\mu\Big).
\end{align*}
In view of \cite[Lemma 12.3]{b-j-j-s-2014-ampa}, the latter equality holds if and only if there exists $c\in(0,\infty)$ such that
\begin{align*}
\text{$1+\hfw^{1-\alpha}+\efw(\hfw^\alpha)\circ\phi^{-1}\hfw^{1-\alpha}\Le c\Big(1+ \efw(\hfw^\alpha)\circ\phi^{-1}\hfw^{1-\alpha}\Big)$ a.e. $[\mu]$.}
\end{align*}
Clearly, this is the same, as to say that there exists $c\in(0,\infty)$ such that
\begin{align*}
\text{$\hfw^{1-\alpha}\Le c\Big(1+ \efw(\hfw^\alpha)\circ\phi^{-1}\hfw^{1-\alpha}\Big)$ a.e. $[\mu]$,}
\end{align*}
which proves (iv)

(v) follows immediately from (iii) and (iv).
\end{proof}
In view of \eqref{dense} and the assertion (i) of Theorem \ref{aluthge} we have.
\begin{cor}\label{perpignon}
Assume \eqref{stand1}. Let $\alpha\in(0,1]$. Suppose that $\cfw$ is densely defined. Then
\begin{align*}
\dz{\dfwa}^\perp&=\chi_{\{\efw(\hfw^\alpha)\circ \phi^{-1}=\infty\}} L^2(\mu),\quad \alpha\in(0,1].
\end{align*}
\end{cor}
The following result sheds some light on the equality $\Delta_{\alpha}(\cfw)=\cfwa$.
\begin{pro}\label{serwis}
Assume \eqref{stand1}. Let $\alpha\in(0,1]$. Suppose that $\cfw$ is densely defined. Consider the following four conditions:
\begin{enumerate}
\item[(i)] $\hfw\Ge c$ a.e. $[\mu]$ for some $c\in(0,\infty)$,
\item[(ii)] $\efw(\hfw^\alpha)\circ\phi^{-1}\Ge c$ a.e. $[\mu]$ on $\{\hfw\neq0\}$ for some $c\in(0,\infty)$,
\item[(iii)] $c\,\efw(\hfw^\alpha)\circ\phi^{-1}\hfw^{1-\alpha}\Ge \hfw^{1-\alpha}$ a.e. $[\mu]$ on $\{\hfw\neq0\}$  for some $c\in(0,\infty)$,
\item[(iv)] $c\Big(1+\efw(\hfw^\alpha)\circ\phi^{-1}\hfw^{1-\alpha}\Big)\Ge \hfw^{1-\alpha}$ a.e. $[\mu]$ on $\{\hfw\neq0\}$  for some $c\in(0,\infty)$.
\end{enumerate}
Then (i) implies (ii), (ii) and (iii) are equivalent, and (iii) implies (iv). If $\hfw\in L^\infty(\mu)$, then (iv) holds. Moreover, if any of the conditions (i) to (iv) is satisfied then $\dfwa=\cfwa$.
\end{pro}
\begin{proof}
Suppose $\hfw\Ge c$ a.e. $[\mu]$ with some $c\in(0,\infty)$. Hence, by \eqref{l2}, \eqref{conexp} and \eqref{fifi} we get\allowdisplaybreaks
\begin{align*}
\int_\sigma \efw(\hfw^\alpha)\circ \phi^{-1}\hfw\D\mu
&=\int_{\phi^{-1}(\sigma)} \efw(\hfw^\alpha)\circ \phi^{-1}\circ \phi\D\mu_w\\
&=\int_{\phi^{-1}(\sigma)} \efw(\hfw^\alpha)\D\mu_w
=\int_{\phi^{-1}(\sigma)} \hfw^\alpha\D\mu_w\\
&\Ge c^\alpha \int_{\phi^{-1}(\sigma)}\D\mu_w
=c^\alpha \int_{\sigma}\hfw \D\mu,\quad \sigma\in\ascr.
\end{align*}
This implies that $\efw(\hfw^\alpha)\circ \phi^{-1}\Ge c^\alpha$ a.e.\ $[\mu]$ on $\{\hfw\neq0\}$ and thus (ii) is satisfied.

It is clear that (ii) and (iii) are equivalent, that (iii) implies (iv), and that (iv) holds whenever $\hfw\in L^\infty(\mu)$. The moreover part follows from Theorem \ref{aluthge}\,(iv).
\end{proof}
It is obvious that the condition (iii) of Proposition \ref{serwis} cannot hold if $\efw(\hfw^\alpha)\circ\phi^{-1}= 0$ on a subset of $\{\hfw\neq0\}$ of positive measure $\mu$. The following example shows that this is actually possible. It also shows that the conditions (iii) and (iv) of Proposition \ref{serwis} are not equivalent.
\begin{exa}
Let $X=\nbb$, $\ascr=2^X$, and $\mu$ be the counting measure on $X$. Let $\phi\colon X\to X$ be given by $\phi(2n)=2n-1$ and $\phi(2n-1)=2n$ for $n\in\nbb$. Suppose $w\colon X\to\rbb_+$. Clearly, $\cfw$ is a well-defined. Since $\phi$ is invertible and $\phi^{-1}=\phi$, $\efw$ acts as the identity operator on $L^2(\mu)$. Consequently, we have $\efw(h)\circ\phi^{-1}=h\circ \phi$ for any $\ascr$-measurable function $h\colon X\to \rbb_+$. In particular, for a given $\alpha\in(0,1)$, $\efw(\hfw^\alpha)\circ\phi^{-1}=\hfw^\alpha\circ \phi$. By \cite[Proposition 79]{b-j-j-sW}, we have
\begin{align}\label{mucha1}
\hfw(x)={\sum_{y\in\phi^{-1}(\{x\})}|w(y)|^2},\quad x\in X,
\end{align}
and thus
\begin{align}\label{bezsensu}
\hfw(n)=\big(w\big(\phi^{-1}(n)\big)\big)^2=\big(w\big(\phi(n)\big)\big)^2,\quad n\in\nbb.
\end{align}
Therefore, we obtain
\begin{align*}
\big(\efw(\hfw^\alpha)\circ\phi^{-1}\,\hfw^{1-\alpha}\big)(n)&=\big(w(n)\big)^{2\alpha} \Big(w\big(\phi(n)\big)\Big)^{2(1-\alpha)},\quad n\in \nbb.
\end{align*}
Assuming that $\sup\{w(n)\colon n\in\nbb\}=\infty$, we deduce from \eqref{bezsensu} and \cite[Proposition 8]{b-j-j-sW} that $\cfw$ is not bounded. Assuming additionally that $w(1)=0$ and $w(2)\neq 0$, we see that $\efw(\hfw^\alpha)\circ\phi^{-1}= 0$ on a subset of $\{\hfw\neq0\}$ of positive measure $\mu$, hence the condition (iii) of Proposition \ref{serwis} is not satisfied. However, the condition (iv) of Proposition \ref{serwis} holds whenever $w(n)\Ge 1$ for $\nbb\setminus\{1,2\}$.
\end{exa}
Note also that the condition (ii) of Proposition \ref{serwis} does not imply (i). The simplest counterexample is with $w=0$ a.e. $[\mu]$. Then $\hfw=0$ a.e. $[\mu]$ and so the condition (ii) follows whilst the condition (i) does not hold. The example below shows that in general (ii) does not imply (i) even if we assume that $\hfw>0$ a.e. $[\mu]$.
\begin{exa}
Let $X=\nbb\times\nbb$, $\ascr=2^X$ and $\mu$ be the counting measure on $X$. Let $\phi\colon X\to X$ be defined by
\begin{align*}
\phi(n,m+1)&=(n,m),\quad n, m \in\nbb,\\
\phi(n+1,1)&=(n,0),\quad n\in\nbb,\\
\phi(1,1)&=(1,1).
\end{align*}
Given a sequence $(a_n)_{n=1}^\infty$ of positive real numbers such that $0$ is its cluster point we define the weight function $w\colon X\to(0,\infty)$ be defined by the below formulas:
\begin{align}\label{w_tree}
\begin{aligned}
w(n,m)&=1,\quad n\in\nbb, m\in \nbb\setminus\{1,2\},\\
w(n,1)&=a_n,\quad n\in\nbb,\\
w(n,2)&=a_{n+1},\quad n\in\nbb.
\end{aligned}
\end{align}
Using \eqref{mucha2}, we obtain
\begin{align}\label{h_tree}
\begin{aligned}
\hfw(n,m)&=1,\quad n\in\nbb, m\in \nbb\setminus\{1\},\\
\hfw(n,1)&=2a_{n+1}^2,\quad n\in\nbb\setminus\{1\},\\
\hfw(1,1)&=a_1^2+2a_2^2.
\end{aligned}
\end{align}
Thus the condition (i) of Proposition \ref{serwis} is not satisfied. On the other hand, by \cite[Proposition 80]{b-j-j-sW}, for every function $f\colon X\to\rbop$ we have
\begin{align*}
\efw(f)(z)=\frac{\sum_{y\in\phi^{-1}(\{x\})}f(y)|w(y)|^2}{\mu_w\big(\phi^{-1}(\{x\})\big)},\quad z\in \phi^{-1}(\{x\}),\, x\in X.
\end{align*}
In particular, we see that
\begin{align}\label{mucha2}
\efw(\hfw^\alpha)\circ \phi^{-1}(x)=\frac{\sum_{y\in\phi^{-1}(\{x\})}\hfw^\alpha(y)|w(y)|^2}{\sum_{y\in\phi^{-1}(\{x\})}|w(y)|^2},\quad x\in X.
\end{align}
By inserting \eqref{w_tree} and \eqref{h_tree} into \eqref{mucha2}, we obtain
\begin{align*}
\efw(\hfw^\alpha)\circ \phi^{-1}(n,m)&= 1,\quad n\in\zbb, m\in \nbb\setminus{1},\\
\efw(\hfw^\alpha)\circ \phi^{-1}(n,1)&=\frac{1+2a_{n+2}^{2\alpha}}{2}, \quad n\in\nbb\setminus\{1\},\\
\efw(\hfw^\alpha)\circ \phi^{-1}(1,1)&=\frac{a_2^2+2a_3^{2\alpha}a_2^2 + a_1^2(a_1^2+2a_2^2)^\alpha}{a_1^2+a_2^2+a_3^2}.
\end{align*}
Since $\frac{1+\theta^2}{2}\Ge \frac12$ for every $\theta\in\rbb$, we see that the condition (ii) Proposition \ref{serwis} holds.
\end{exa}
Below we discuss the relation between $\cfwa$ and $\dfwa$ in the case of assorted composition operators induced by linear transformations of $\rbb^n$. In particular, we provide an example of a composition operator $C_\phi$ such that its Alutghe transform $\Delta_{\frac12}(C_\phi)$ differs from the associated composition operator $C_{\phi, \mathbf{1}_{\frac{1}{2}}}$.
\begin{exa}\label{properinclusion}
Let $n\in \nbb$. Let $X=\rbb^n$ and $\ascr=\borel{\rbb^n}$. Suppose $\rho(z) = \sum_{k=0}^\infty a_k z^k$, $z \in \cbb$, is an entire function such that $a_n$ is non-negative for every $k\in\zbb_+$ and $a_{k_0} > 0$ for some $k_0 \Ge 1$. Let $\mu=\mu^\rho$ be the $\sigma$-finite measure on $\ascr$ given by
\begin{align*}
\mu^\rho(\sigma) = \int_\sigma \rho\big(\|x\|^2\big) \D\M_n(x),\quad \sigma\in\borel{\rbb^n},
\end{align*}
where $\M_n$ is the $n$-dimensional Lebesgue measure on $\rbb^n$ and $\|\cdot\|$ is a norm on $\rbb^n$ induced by an inner product. Let $\phi\colon \rbb^n\to\rbb^n$ be invertible and linear. Clearly, $\phi$ induces a composition operator $C_\phi$ in $L^2(\mu)$. Changing the variables, we may assume that
\begin{align} \label{rn-matrix}
\mathsf{h}_\phi(x)= \frac{1}{|\det \phi|} \frac{\rho(\|\phi^{-1}(x)\|^2)}{\rho(\|x\|^2)}, \quad x \in \rbb^n \setminus \{0\}.
\end{align}
Hence, by \eqref{dense}, $C_\phi$ is densely defined. The boundedness of $C_\phi$ has been fully characterized in \cite[Proposition 2.2]{sto-1990-hmj} in the following way:
\begin{align}\label{shok}
\begin{minipage}{80ex}
if $\rho$ is a polynomial, then $C_\phi$ is bounded; if $\rho$ is not a polynomial, then $C_\phi$ is bounded if and only if  $\|\phi^{-1}\|\Le 1$.
\end{minipage}
\end{align}
Therefore, if $\rho$ is polynomial and/or $\|\phi^{-1}\|\Le 1$, then, by Theorem \ref{aluthge}, both $\Delta_\alpha(C_\phi)$ and $C_{\phi,\mathbf{1}_\alpha}$ are bounded and coincide. If $\rho$ is not a polynomial and $\|\phi^{-1}\|>1$, then $C_\phi$ is not bounded and we have two possibilities: either $\|\phi\|\Le 1$, or $\|\phi\|> 1$. In the first case, by \eqref{shok}, $C_{\phi^{-1}}$ is bounded. The latter, according to \eqref{boundedness} and \eqref{rn-matrix}, implies that
\begin{align*}
\sup_{x\in\rbb^n\setminus\{0\}}\frac{\rho\big(\|\phi(x)\|^2\big)}{\rho\big(\|x\|^2\big)}<\infty,
\end{align*}
or equivalently
\begin{align*}
\mathsf{h}_\phi\Ge c\quad  \text{a.e. $[\mu]$}
\end{align*}
with some $c\in(0,\infty)$. Therefore, by Proposition \ref{serwis}\,(i), $\Delta_{\alpha}(C_\phi)=C_{\phi,\mathbf{1}_\alpha}$. On the other hand, if $\|\phi\|>1$, then, by \eqref{shok}, $C_{\phi^{-1}}$ is not bounded. In this case, it may happen then that $\dfwa\neq C_{\phi,\mathbf{1}_\alpha}$. Indeed, it suffices to consider $\phi$ such that $\phi^2$ is the identity (keeping all the other assumptions). Then, since $\phi$ is invertible, the conditional expectation $\mathsf{E}_{\phi}:=\mathsf{E}_{\phi, 1}$ acts as the identity and thus using \eqref{rn-matrix} we get\allowdisplaybreaks
\begin{align*}
\Big(\mathsf{E}_{\phi}(\mathsf{h}_{\phi}^{1/2})\circ \phi^{-1}\Big) (x)\cdot \mathsf{h}_\phi^{1/2}(x)
&=\Big(\mathsf{h}_{\phi}^{1/2}\circ \phi^{-1}\Big) (x)\cdot \mathsf{h}_\phi^{1/2} (x)\\
&=\frac{1}{|\det \phi|}\Bigg(\frac{\rho\big(\|\phi^{-2}(x)\|^2\big)}{\rho\big(\|\phi^{-1}(x)\|^2\big)} \frac{\rho\big(\|\phi^{-1}(x)\|^2\big)}{\rho\big(\|x\|^2\big)}\Bigg)^{1/2}\\
&=\frac{1}{|\det \phi|},\quad x\in\rbb^n\setminus\{0\}.
\end{align*}
Since $C_\phi$ is not bounded, the Radon-Nikodym derivative $\mathsf{h}_\phi^{1/2}$ is not essentially bounded by \eqref{boundedness}, hence, in view of the above, there is no positive real number $c$ such that
\begin{align*}
\mathsf{h}_\phi^{1/2}\Le c \Big(1+\big(\mathsf{E}_{\phi}(\mathsf{h}_{\phi}^{1/2})\circ \phi^{-1}\big)\cdot \mathsf{h}_\phi^{1/2}\Big)\quad\text{a.e. $[\mu]$.}
\end{align*}
This, according to Theorem \ref{aluthge}, is the same as to say that $\Delta_{\frac12}(C_\phi)\neq C_{\phi,\mathbf{1}_{\frac12}}$.
\end{exa}
There are known examples of densely defined weighted shifts on directed trees or densely defined composition operators in $L^2$-spaces whose $\alpha$-Aluthge transforms is not densely defined (see \cite{tre-2015-jmaa} for even more sophisticated ones). Below we discuss in more detail the circumstances of when $\cfw$ acting in $L^2$-space with respect to the counting measure is densely defined while its $\alpha$-Aluthge transform is not.
\begin{exa}
Let $X$ be a countable set, $\mu$ be the counting measure on $2^X$, $w\colon X\to (0,\infty)$, $\phi\colon X\to X$, and $\alpha\in(0,1]$. In view of \eqref{dziedzina} and \eqref{mucha1}, we have
\begin{align}\label{da1}
\overline{\dz{\cfw}}=L^2(\mu)\quad\Longleftrightarrow\quad \sum_{y\in\phi^{-1}(\{x\})}|w(y)|^2<\infty,\ x\in X.
\end{align}
Assume that $\sum_{y\in\phi^{-1}(\{x\})}|w(y)|^2<\infty$ for every $x\in X$.
By Corollary \ref{perpignon} and \eqref{mucha2}, that
\begin{align}\label{lille}
\dz{\cfwa}^\perp=\chi_{\{x\in X\colon \sum_{y\in\phi^{-1}(\{x\})}(\sum_{z\in\phi^{-1}(\{y\})}|w(z)|^2)^\alpha|w(y)|^2=\infty\}}L^2(\mu).
\end{align}

Thus we see that constructing an example of a densely defined $\cfw$ such that $\dfwa$ is not densely defined amounts to finding a transformation $\phi$ and a weight function $w$ satisfying the following two conditions
\begin{enumerate}
\item $\sum_{y\in\phi^{-1}(\{x\})}|w(y)|^2<\infty$ for every $x\in X$,
\item $\sum_{y\in\phi^{-1}(\{x\})}\Big(\sum_{z\in\phi^{-1}(\{y\})}|w(z)|^2\Big)^\alpha|w(y)|^2=\infty$ for some $x\in X$.
\end{enumerate}
Replacing the condition (2) above with
\begin{enumerate}
\item[(2$^\prime$)] $\sum_{y\in\phi^{-1}(\{x\})}\Big(\sum_{z\in\phi^{-1}(\{y\})}|w(z)|^2\Big)^\alpha|w(y)|^2=\infty$ for every $x\in X$,
\end{enumerate}
and leaving (1) leads to an example of a densely defined $\cfw$  such that $\dfwa$ has trivial domain. Both the situations are possible (see \cite[Theorem 6.6]{tre-2015-jmaa}). It is worth noting that the triviality of the domain of $\dfwa$ implies that $\phi^{-1}(\{x\})$ has to be infinite for every $x\in X$.
\end{exa}
Below we address the question of dense definiteness of $\cfwa$. First, we note that applying \eqref{dense}, Lemma \ref{relacje}, Theorem \ref{aluthge}, Corollary \ref{perpignon}, and \eqref{fifi} gives the following.
\begin{cor}
Suppose \eqref{stand1} holds, $\cfw$ is densely defined, and $\alpha\in (0,1]$. Then the following conditions are equivalent$:$
\begin{itemize}
\item[(i)] $\cfwa$ is densely defined
\item[(ii)] $\efw(\hfw^\alpha)\circ \phi^{-1}<\infty$ a.e. $[\mu]$,
\item[(iii)] $\efw(\hfw^\alpha)<\infty$ a.e. $[\mu]$ on $\{w\neq 0\}$.
\end{itemize}
\end{cor}
Next, we provide sufficient conditions for dense definiteness of $\cfwa$ written in terms of powers of $\cfw$. We precede them with an auxiliary lemma.
\begin{lem}\label{surprise}
Assume \eqref{stand1}. Let $\alpha\in(0,1)$, $f\colon X\to\rbop$ be $\ascr$-measurable, and $\hfw<\infty$ a.e. $[\mu]$. Suppose that $\mu\big(\big\{\efw(f^\alpha)\circ\phi^{-1}=\infty\big\}\big)>0$. Then $\mu\big(\big\{\efw(f)\circ\phi^{-1}=\infty\big\}\big)>0$.
\end{lem}
\begin{proof}
Fix $\alpha\in(0,1)$ and assume that $\efw(f^\alpha)\circ\phi^{-1}=\infty$ on a set $\varOmega\in \ascr$ such that $\mu(\varOmega)>0$. Clearly, $\mu\big(\varOmega\cap \{\hfw\neq 0\} \big)>0$ so me may assume, without loss of generality, that $\varOmega\subseteq\{\hfw\neq 0\}$. Since $\mu$ is $\sigma$-finite and $\hfw<\infty$ a.e. $[\mu]$, we may also assume that $\mu(\varOmega)<\infty$ and that $\hfw<c$ a.e. $[\mu]$ for some $c\in(0,\infty)$. As a consequence, $\int_\varOmega\, \hfw \D\mu<\infty$. Set $\varDelta=\phi^{-1}(\varOmega)$. Then
\begin{align*}
\mu_w(\varDelta)=\int_X\chi_\varOmega\circ\phi \D\mu_w=\int_\varOmega \,\hfw <\infty.
\end{align*}
Moreover, by \eqref{fifi},  $\efw\big(f^\alpha\big)(x)=\infty$ for $\mu_w$-a.e. $x\in \varDelta$. Now, let $\widetilde\varOmega\in\ascr$ be any subset of $\varOmega$ such that $0<\mu(\widetilde\varOmega)<\infty$. Set $\widetilde\varDelta=\phi^{-1}\big(\widetilde\varOmega\big)$. Then $\mu_w\big(\widetilde\varDelta\big)\Le \mu_w(\varDelta)<\infty$. We also see that $\mu_w\big(\widetilde\varDelta\big)>0$. Indeed, otherwise by \eqref{l2} we have
\begin{align*}
0=\int_{\widetilde\varDelta}\D\mu_w=\int_{\widetilde\varOmega}\hfw\D\mu
\end{align*}
which contradicts the inclusion $\tilde \varOmega\subseteq\{\hfw>0\}$. Hence, by \eqref{conexp}, we have
\begin{align}\label{mumford}
\int_X \chi_{\widetilde\varDelta}\, f^\alpha\D\mu_w=\int_X \chi_{\widetilde\varDelta}\, \efw\big(f^\alpha\big)\D\mu_w=\infty
\end{align}
Since
\begin{align*}
\int_X\chi_{\widetilde\varDelta\cap \{f\Le1\}}f^\alpha\D\mu_w\Le \int_X\chi_{\widetilde\varDelta}\D\mu_w=\mu_w\big(\widetilde\varDelta\big)<\infty,
\end{align*}
we deduce from \eqref{mumford} that
\begin{align*}
\int_X\chi_{\widetilde\varDelta\cap \{f>1\}}\,f\D\mu_w \Ge \int_X\chi_{\widetilde\varDelta\cap \{f>1\}}\,f^\alpha\D\mu_w=\infty.
\end{align*}
This in turn, again by \eqref{l2} and \eqref{fifi}, implies that
\begin{align*}
\int_X\chi_{\widetilde\varOmega}\,\efw(f)\circ \phi^{-1}\hfw\D\mu=\int_X\chi_{\widetilde\varDelta}\,\efw(f)\D\mu_w=\int_X \chi_{\widetilde\varDelta}\, f\D\mu_w=\infty.
\end{align*}
Because $\widetilde\varOmega$ was chosen arbitrarily, we deduce that (ii) have to be satisfied.
\end{proof}
\begin{pro}\label{zimno}
Assume \eqref{stand1}. Let $\alpha\in(0,1]$. If $\cfw^k$ is densely defined for some $k\in\nbb\setminus\{1\}$, then $\cfwa$ is densely defined.
\end{pro}
\begin{proof}
Clearly, it suffices to consider the case $k=2$, thus we assume that $\cfw^2$ is densely defined. Then in particular $\cfw$ is densely defined, or equivalently, by \eqref{dense}, $\hfw<\infty$ a.e. $[\mu]$.

Suppose contrary to our claim that $\cfwa$ is not densely defined. Then, by \eqref{dense}, $\hfwa=\infty$ on a set of positive measure $\mu$. This, by Lemma \ref{relacje}\,(ii), implies that there exist a set $\varOmega\in\ascr$ such that $0<\mu(\varOmega)<\infty$ and $\varOmega\subseteq\big\{\efw\big(\hfw^\alpha\big)\circ\phi^{-1}=\infty\big\}\cap \{\hfw\neq0\}$. Therefore, by Lemma \ref{surprise}, there exists $\varOmega^\prime\in\ascr$ such that $\varOmega^\prime\subseteq\{\efw(\hfw)\circ\phi^{-1}=\infty\}\cap\{\hfw\neq 0\}$ and $\mu(\varOmega^\prime)>0$. Thus, in view of \cite[Lemmata 26 and 44]{b-j-j-sW}, $\cfw^2$ is not densely defined. This contradiction completes the proof.
\end{proof}
The following example shows that the condition of dense definiteness of $\cfw^2$ is sufficient but not necessary for the dense definiteness of $\cfwa$ with $\alpha\in(0,1)$. On the other hand, by \cite[Lemmata 26 and 43]{b-j-j-sW}, $\cfw^2$ is densely defined if and only if $C_{\phi,w_1}$ is densely defined.
\begin{exa}\label{buda}
Let $X:=\zbb_+\cup\nbb\times\nbb$ and $\phi\colon X\to X$ is given by
\begin{align*}
\phi(k-1)&:=k,\quad k\in\nbb,\\
\phi\big((k,1)\big)&:=0,\quad k\in\nbb,\\
\phi\big((m,n)\big)&:=(m,n-1),\quad m\in\nbb,\ n\in\nbb\setminus\{1\}.
\end{align*}
Let $\alpha\in(0,1)$. Let $w\colon X\to(0,\infty)$ be any function satisfying the following three conditions\allowdisplaybreaks
\begin{align*}
&\sum_{k=1}^\infty \big(w((k,1))\big)^2<\infty,\\
&\sum_{k=1}^\infty \big(w((k,1))\big)^2\big(w((k,2))\big)^{2\alpha}<\infty,\\
&\sum_{k=1}^\infty \big(w((k,1))\big)^2\big(w((k,2))\big)^2=\infty
.
\end{align*}
For example any $w\colon X\to(0,\infty)$ with $w((k,1))=\frac{1}{k}$ and $w((k,2))=\sqrt{k}$ for $k\in\nbb$ does the job. Let $\mu$ be the counting measure on $2^X$ and $\alpha\in(0,\infty)$. Then, by \eqref{mucha1} and \eqref{mucha2}, we obtain
\begin{align*}
&0<\hfw(x)<\infty,\quad x\in X,\\
&\efw(\hfw^\alpha)\circ \phi^{-1}(x)<\infty,\ x\in X,\\
&\efw(\hfw)\circ \phi^{-1}(0)=\infty.
\end{align*}
Consequently, by \eqref{dziedzina} and Lemma \ref{relacje}\,(ii), both the operators $\cfw$ and $\cfwa$ are densely defined, while, in view of \cite[Lemmata 26 and 43]{b-j-j-sW}, $\cfw^2$ is not densely defined.
\end{exa}
Employing \cite[Theorem 45]{b-j-j-sW} we get a simplified version of Proposition \ref{zimno} if $w$ is nonzero (note that we use a different notation below than in \cite{b-j-j-sW}).
\begin{cor}\label{zimno2}
Assume \eqref{stand1}. Let $\alpha\in(0,1]$. If $w\neq0$ a.e. $[\mu]$ and $\hsf_{\phi^k, w^{[k]}}<\infty$ a.e. $[\mu]$ for some $k\in\nbb\setminus\{1\}$, where $ w^{[k]}:=\prod_{j=1}^{k-1}w\circ\phi^j$, then $\cfwa$ is densely defined.
\end{cor}
The following example shows that the assumption ``$\cfw^k$ is densely defined for some $k\in\nbb\setminus\{1\}$'' in Proposition \ref{zimno} cannot be in general replaced by ``$C_{\phi^k, w^{[k]}}$ is densely defined for some $k\in\nbb\setminus\{1\}$''.
\begin{exa}
Let $(X,\ascr,\mu)$, $\phi\colon X\to X$, and $w\colon X\to\cbb$ are such that \eqref{stand1} is satisfied, $\cfw$ is densely defined, $C_{\phi^2, w^{[2]}}$ is not densely defined, and $C_{\phi^3, w^{[3]}}$ is densely defined (see \cite[Example 46]{b-j-j-sW}). By \cite[Lemma 44]{b-j-j-sW}, $\cfw^2$ is not densely defined. Moreover, by \cite[Lemma 26]{b-j-j-sW}, $\efw(\hfw)\circ\phi^{-1}=\infty$ on a subset of $\{\hfw\neq 0\}$ of positive measure $\mu$. This, according to Lemma \ref{surprise},  yields that $\efw(\hfw^\alpha)\circ\phi^{-1}=\infty$ on a subset of $\{\hfw\neq 0\}$ of positive measure $\mu$ for any $\alpha\in(0,1)$. Therefore, by Lemma \ref{relacje}\,(ii) and \eqref{dense} $\cfwa$ is not densely defined for any $\alpha\in(0,1)$.
\end{exa}
\section{$p$-hyponormality}
In this section we investigate $p$-hyponormality of weighted composition operators and their Aluthge transforms. We begin with two auxiliary results the first of which seems to be folklore.
\begin{lem}\label{folklor}
Let $\hh$ be a complex Hilbert space, $A$ be a normal operator in $\hh$, and $P$ be an orthogonal projection on $\hh$. Assume that $AP$ is normal and $PA\subseteq AP$. Then for every Borel function $\Phi\colon\cbb\to\cbb$ we have\footnote{Both the expressions $\Phi(A)$ and $\Phi(AP)$ are understood in terms of the functional calculus.} $\Phi(A)P=\Phi(AP)$.
\end{lem}
\begin{proof}
Let $E$ be the spectral measure of $A$. Since $PA\subseteq AP$, we have $EP=PE$ and
\begin{align*}
\is{APf}{g}=\int_\cbb t\is{E(\D t)Pf}{g}=\int_\cbb t\is{EP(\D t)f}{g},\quad f\in\dz{AP}, g\in\hh,
\end{align*}
which means that $EP$ is a spectral measure of $AP$. In the same manner we get
\begin{align*}
\int_\cbb |\Phi(t)|^2\is{EP(\D t)f}{f}=\int_\cbb |\Phi(t)|^2\is{E(\D t)Pf}{Pf},\quad f\in\hh.
\end{align*}
Hence for $f\in\hh$ we have $f\in\dz{\Phi(AP)}$ if and only if $Pf\in \dz{\Phi(A)}$. This and
\begin{multline*}
\is{\Phi(AP)f}{g}
=\int_\cbb \Phi(t)\is{EP(\D t)f}{g}
=\int_\cbb \Phi(t)\is{E(\D t)Pf}{g}\\
=\is{\Phi(A)Pf}{g},\quad f\in \dz{\Phi(AP)},\ g\in\hh,
\end{multline*}
yield the claim.
\end{proof}
Before the next lemma we recall that the mapping
\begin{align*}
L^2(\mu)\ni f\mapsto f_w\in L^2(\mu_w),
\end{align*}
where $f_w=\chi_{\{w\neq 0\}}\frac{f}{w}$, is a well-defined contraction and $\efw$ is a linear contraction on $L^2(\mu_w)$ (cf. \cite[(17)]{b-j-j-sW}). 
\begin{lem}\label{projekcja}
Assume \eqref{stand1}. Then the formula
\begin{align}\label{projekcja+}
Pf=w\cdot\efw\big(f_w\big),\quad f\in L^2(\mu),
\end{align}
defines an orthogonal projection $P$ on $L^2(\mu)$.
\end{lem}
\begin{proof}
First we observe that $P$ is a bounded operator on $L^2(\mu)$, which easily follows from the preceding remark. Now we prove that $P^2=P^*=P$. This follows immediately from the equality $P=P^*P$, which in turn can be deduced from
\begin{align*}
\is{Pf}{g}&=\int_X w\cdot \efw(f_w)\cdot \overline{g}\D\mu=\int_X \efw(f_w)\cdot \overline{g_w}\D\mu_w\\
&=\int_X \efw(f_w)\cdot \overline{\efw(g_w)}\D\mu_w=\is{Pf}{Pg},\quad f,g\in L^2(\mu).
\end{align*}
Hence the proof is complete.
\end{proof}
Combining the above two results with the description of the adjoint of a weighted composition operator given in Theorem \ref{adjoint} enables us to write down the formula for any positive power of the modulus of the adjoint of a weighted composition operator.
\begin{thm}\label{modadj}
Assume \eqref{stand1}. Suppose $\cfw$ is densely defined. Then the following equalities
\begin{align*}
\dz{|\cfw^*|^p}&=\big\{f\in L^2(\mu)\colon w\cdot (\hfw\circ \phi)^{p/2}\efw(f_w)\in L^2(\mu)\big\},\\
|\cfw^*|^p f&=w\cdot (\hfw\circ \phi)^{p/2}\efw(f_w),\quad f\in \dz{|\cfw^*|},
\end{align*}
hold for every $p\in(0,\infty)$.
\end{thm}
\begin{proof}
Apply Theorem \ref{adjoint} and Lemmata \ref{folklor} and \ref{projekcja} with $A=M_{(\hfw\circ \phi)^{1/2}}$, the operator of multiplication by $(\hfw\circ\phi)^{1/2}$ in $L^2(\mu)$, a projection $P$ given by \eqref{projekcja+}, and $\Phi(t)=t^p$.
\end{proof}
Before giving a characterization of $p$-hyponormal weighted composition operators we need to recall the following lemma, which is an adaptation of \cite[Lemma 2.3]{b-j-l-2005-jot} to our context. We include a modified proof for the readers's convenience (it incorporates partially Herron's results from \cite{her-2011-oam} on contractive weighted conditional expectation operators).
\begin{lem}\label{herron}
Assume \eqref{stand1}. Suppose $\hfw<\infty$ a.e. $[\mu]$. Let $g_1,g_2\colon X\to[0,\infty)$ be $\ascr$-measurable. Then the following two conditions are equivalent$:$
\begin{enumerate}
\item[(i)] for every $\ascr$-measurable $f\colon X\to \rbop$,
\begin{align}\label{phypint-}
\int_X \efw^2\Big(( g_1 f)^{\frac12}\Big)\D\mu_w \Le \int_X g_2f\D\mu_w,
\end{align}
\item[(ii)] $\mu_w\big(\{g_1\neq 0\}\setminus \{g_2\neq 0\}\big)=0$ and $\efw\Big(\chi_{\{g_2\neq 0\}}\frac{g_1}{g_2}\big)\Le 1$ a.e. $[\mu_w]$.
\end{enumerate}
\end{lem}
\begin{proof}
(i)$\Rightarrow$(ii) Set $\varOmega=\{g_1\neq 0\}\setminus \{g_2\neq 0\}$. Substituting $f=\chi_\varOmega\cdot\frac{1}{g_1}$ into \eqref{phypint-} we get
\begin{align*}
\int_X \efw^2(\chi_{\varOmega})\D\mu_w=0.
\end{align*}
This implies that $\efw(\chi_{\varOmega})=0$ a.e. $[\mu_w]$ and so we have
\begin{align*}
0=\int_{\phi^{-1}(\sigma)}\efw(\chi_{\varOmega})\D\mu_w=\int_{\phi^{-1}(\sigma)}\chi_{\varOmega}\D\mu_w,\quad \sigma\in\ascr.
\end{align*}
Thus $\mu_w(\varOmega)=0$.

Next, using \eqref{conexp+} and \eqref{phypint-} with $f=\chi_\sigma \circ\phi\cdot\chi_{\{g_2\neq 0\}}\frac{g_1}{g_2^2}$, we get
\begin{align*}
\int_X \chi_\sigma \circ\phi\cdot\efw^2(g)\D\mu_w\Le\int_X \chi_\sigma \circ\phi\cdot g\D\mu_w=\int_X \chi_\sigma\circ\phi\cdot\efw(g)\D\mu_w,\quad \sigma\in\ascr,
\end{align*}
with $g:=\chi_{\{g_2\neq 0\}}\frac{g_1}{g_2}$. Hence, by \eqref{l2} and
\eqref{fifi}, we get
\begin{align*}
\int_X \chi_\sigma \cdot \efw^2(g)\circ \phi^{-1}\cdot\hfw \D\mu\Le\int_X \chi_\sigma\cdot \efw(g)\circ\phi^{-1} \cdot\hfw\D\mu,\quad \sigma\in\ascr.
\end{align*}
Therefore, we obtain
\begin{align*}
\efw^2(g)\circ\phi^{-1}\cdot\hfw\Le\efw(g)\circ\phi^{-1}\cdot \hfw\quad \text{a.e. $[\mu]$},
\end{align*}
which yields
\begin{align*}
\efw^2(g)\circ\phi^{-1}\Le\efw(g)\circ\phi^{-1}\quad \text{a.e. $[\mu]$ on $\{\hfw\neq 0\}$}.
\end{align*}
In view of the fact that $\hfw\circ\phi>0$ a.e. $[\mu_w]$, the above and \eqref{fifi} imply that the inequality $\efw(g)\Le 1$ a.e. $[\mu_w]$ is satisfied.

(ii)$\Rightarrow$(i) Assuming that $\efw(g)\Le 1$ a.e. $[\mu_w]$, where $g:=\chi_{\{g_2\neq 0\}}\frac{g_1}{g_2}$, and using the conditional H\"{o}lder inequality (see \cite[Lemma A.1]{b-j-j-sW}) we get
\begin{align}\label{bbr}
\int_X \efw^2 \Big(\big(g \tilde f\big)^{\frac12}\Big)\D\mu_w\Le \int_X \efw(g)\cdot \efw\big(\tilde f\big)\D\mu_w \Le \int_X \tilde f\D\mu_w
\end{align}
which holds for every $\ascr$-measurable function $\tilde f\colon X\to \rbop$. Now, substituting $\tilde f=g_2f$ into \eqref{bbr} and using the fact that $\chi_{\{g_2\neq0\}}g_1f=g_1f$ a.e. $[\mu_w]$, we see that \eqref{phypint-} is satisfied. This completes the proof.
\end{proof}
We are now in the position to prove the aforementioned characterization of unbounded $p$-hyponormal weighed composition operators (compare it with a recent characterization of hyponormal weighted composition operators in \cite[Theorem 53]{b-j-j-sW}). It is worth mentioning that $p$-hyponormality of bounded weighted composition operators was studied in \cite{azi-jab-bkms-2010} under restrictive assumptions mentioned in the Introduction.
\begin{thm}\label{p-hyp}
Assume \eqref{stand1}. Suppose $\cfw$ is densely defined. Let $p\in(0,\infty)$. Then the following four assertions are equivalent$:$
\begin{enumerate}
\item[(i)] $\cfw$ is $p$-hyponormal,
\item[(ii)] for every $f\in L^2(\mu)$,
\begin{align}\label{phypint}
\int_X\hfw^p\circ\phi\cdot |\efw(f_w)|^2\D\mu_w\Le \int_X \hfw^p\cdot |f|^2\D\mu,
\end{align}
\item[(iii)] for every $\ascr$-measurable $f\colon X\to\rbop$,
\begin{align}\label{phypint++}
\int_X\hfw^p\circ\phi\cdot \efw^2(f)\D\mu_w\Le \int_X \hfw^p\cdot f^2\D\mu_w,
\end{align}
\item[(iv)] $\hfw>0$ a.e.\ $[\mu_w]$ and $\efw\Big(\frac{\hfw^p\circ \phi}{\hfw^p}\Big)\Le 1$ a.e.\ $[\mu_w]$.
\end{enumerate}
\end{thm}
\begin{proof}
(i) $\Rightarrow$ (ii) By Theorems \ref{polar} and \ref{modadj} for every $f\in \dz{|\cfw|^p}$ both the integrals in \eqref{phypint} are finite and the inequality in \eqref{phypint} is satisfied. For other $f$'s in $L^2(\mu)$, by Theorem \ref{polar}, the right-hand side of \eqref{phypint} is infinite, hence the inequality in \eqref{phypint} holds as well.

(ii) $\Rightarrow$ (i) In view of Theorems \ref{polar} and \ref{modadj}, $\dz{|\cfw|^p}\subseteq \dz{|\cfw^*|^p}$, and $\||\cfw^*|^pf\|\Le\||\cfw|^pf\|$  for every $f\in \dz{|\cfw|^p}$.

(ii) $\Rightarrow$ (iii) Using standard measure-theoretic arguments we may show that (ii) implies
\begin{align*}
\int_X\hfw^p\circ\phi\cdot \efw^2\big(\chi_{\{w\neq 0\}} f\big)\D\mu_w\Le \int_X \hfw^p\cdot f^2\D\mu_w
\end{align*}
for every $\ascr$-measurable $f\colon X\to \rbop$. Now it suffices to use the equality
\begin{align}\label{wakeup}
\efw(f)=\efw\big(\chi_{\{w\neq 0\}} f\big)\quad \text{a.e. $[\mu_w]$}
\end{align}
which is valid for every $\ascr$-measurable $f\colon X\to \rbop$ by \eqref{conexp}.

(iii) $\Rightarrow$ (iv) This follows from Lemma \ref{herron} applied with $g_1=\hfw^p\circ\phi$ and $g_2=\hfw^p$ (recall that $\hfw\circ\phi>0$ a.e. $[\mu_w]$).

(iv) $\Rightarrow$ (ii) In view of \cite[assertion (A.11)]{b-j-j-sW}, the condition (ii) is equivalent to
\begin{align*}
\int_X\hfw^p\circ\phi\cdot \efw^2\big(|f|_{|w|}\big)\D\mu_w\Le \int_X \hfw^p\cdot |f|^2\D\mu,\quad f\in L^2(\mu).
\end{align*}
Thus it suffices to use Lemma \ref{herron} with $g_1=\hfw^p\circ\phi$ and $g_2=\hfw^p$ again. This completes the proof.
\end{proof}
Combining the above result with the conditional Jensen's inequality we show using purely measure theoretic tools that $p$-hyponormal weighted composition operator is $q$-hyponormal for any $0<q<p$. Note that using the Heinz inequality (see \cite[Proposition 10.14]{sch}) one can even prove a more general result: any $p$-hyponormal operator in a complex Hilbert space is $q$-hyponormal for $0<q<p$. 
\begin{pro}\label{pq}
Assume \eqref{stand1}. Let $p\in(0,\infty)$. Suppose $\cfw$ is $p$-hyponormal. Then $\cfw$ is $q$-hyponormal for every $q\in(0,p)$.
\end{pro}
\begin{proof}
According to conditional Jensen's inequality (cf. \cite[Lemma A.1]{b-j-j-sW}) and Theorem \ref{p-hyp} we have
\begin{align*}
\efw^{\frac{p}{q}}\bigg(\frac{\hfw^q\circ\phi}{\hfw^q}\bigg)
\Le \efw\bigg(\frac{\hfw^p\circ\phi}{\hfw^p}\bigg) \Le 1 \quad \text{a.e. $[\mu_w]$}
\end{align*}
for any $q\in(0,p)$. Thus, by Theorem \ref{p-hyp} we get $q$-hyponormality of $\cfw$.
\end{proof}
One of the important properties the Aluthge transformation of bounded operators has is that it improves $p$-hyponormality. We show below that this is shared by unbounded weighted composition operators (see Theorem \ref{ptaszki} below). Before that we consider a class of weighted composition operators larger than that of $p$-hyponormal ones, namely, the one described by condition (i) of Proposition \ref{phq} below. It is related to the class of $p$-hyponormal weighed composition operators in a way that resemble the relation between quasinormal and normal operators.
\begin{pro}\label{phq}
Assume \eqref{stand1}. Suppose $\cfw$ is densely defined. Let $p\in(0,\infty)$. Then the following two conditions are equivalent$:$
\begin{enumerate}
\item[(i)] $\dz{|\cfw|^p \cfw}\subseteq \dz{|\cfw^*|^p\cfw}$ and $\||\cfw^*|^p\cfw f\|\Le\||\cfw|^p\cfw f\|$ for every $f\in \dz{|\cfw|^p \cfw}$,
\item[(ii)] $\hfw^p\circ\phi\Le \efw(\hfw^p)$ a.e. $[\mu_w]$.
\end{enumerate}
\end{pro}
\begin{proof}
Let $\ff=\dz{|\cfw|^{p} \cfw}$ and $\mathcal{G}=\dz{|\cfw^*|^{p} \cfw}$. In view of \eqref{l2}, \eqref{conexp}, and \eqref{fifi} we have
\begin{align*}
\int_X \hfw^{p} |f\circ \phi|^2\D\mu_w&=\int_X \efw\big(\hfw^{p}\big)\cdot |f\circ \phi|^2\D\mu_w\\
&=\int_X \hfw\cdot\efw\big(\hfw^{p}\big)\circ \phi^{-1}\cdot |f|^2\D\mu,\quad f\in L^2(\mu),
\end{align*}
and using \eqref{l2} and  \eqref{wakeup} we get
\begin{align*}
\int_X(\hfw\circ\phi)^p\cdot\efw^2\big(\chi_{\{w\neq0\}}\big)\cdot|f\circ\phi|^2\D\mu=\int_X \hfw^{p+1}\cdot |f|^2\D\mu,\quad f\in L^2(\mu).
\end{align*}
Combining this with \eqref{dziedzina}, Theorem \ref{polar}, and Theorem \ref{modadj}  we obtain
\begin{align}\label{hive+}
\begin{aligned}
\ff&=L^2\Big(\big(1+\hfw+\hfw\cdot\efw(\hfw^{p})\circ\phi^{-1}\big)\D\mu\Big),\\
\mathcal{G}&=L^2\Big(\big(1+\hfw+\hfw^{p+1}\big)\D\mu\Big).
\end{aligned}
\end{align}

(i) $\Rightarrow$ (ii) Since
\begin{align}\label{hive++}
\begin{aligned}
\||\cfw|^p\cfw f\|^2&=\int_X \hfw\cdot\efw(\hfw^{p})\circ \phi^{-1} |f|^2\D\mu,\quad f\in\ff\\
\||\cfw^*|^p\cfw g\|^2&=\int_X \hfw^{p+1} |g|^2\D\mu,\quad g\in\mathcal{G},
\end{aligned}
\end{align}
we see that
\begin{align}\label{hive--}
\int_X \hfw^{p+1} |f|^2\D\mu\Le \int_X \hfw\cdot\efw(\hfw^{p})\circ \phi^{-1} |f|^2\D\mu,\quad f\in\ff.
\end{align}
Hence we deduce
\begin{align*}
\hfw^{p+1} \Le \hfw\cdot\efw(\hfw^{p})\circ \phi^{-1} \quad \text{a.e. $[\mu]$ on a set $\big\{\efw(\hfw^p)\circ\phi^{-1}<\infty\big\}$}
\end{align*}
which implies
\begin{align*}
\hfw^{p+1} \Le \hfw\cdot\efw(\hfw^{p})\circ \phi^{-1}\cdot \quad \text{a.e. $[\mu]$.}
\end{align*}
This and \eqref{fifi} imply
\begin{align}\label{hive-}
\hfw^{p}\circ \phi \Le \efw(\hfw^{p})\quad \text{a.e. $[\mu_w]$}.
\end{align}

(ii) $\Rightarrow$ (i) Assuming \eqref{hive-} we  deduce \eqref{hive--} which, together with \eqref{hive+}, yields the inclusion $\ff\subseteq \mathcal{G}$. Consequently, it follows from \eqref{hive++} that $\||\cfw^*|^p\cfw f\|\Le\||\cfw|^p\cfw f\|$ for every $f\in\ff$. This completes the proof.
\end{proof}
Clearly, if $\cfw$ is $p$-hyponormal, then it belongs to the $p$-th class above. It turns out that $p$-hyponormality is even stronger property.
\begin{pro}\label{phq+}
Assume \eqref{stand1}. Suppose $\cfw$ is $p$-hyponormal for some $p\in(0,\infty)$. Then for every $q\in(0,\infty)$, $\dz{|\cfw|^q \cfw}\subseteq \dz{|\cfw^*|^q\cfw}$ and $\||\cfw^*|^q\cfw f\|\Le\||\cfw|^q\cfw f\|$ for every $f\in \dz{|\cfw|^q \cfw}$.
\end{pro}
\begin{proof}
First we note that due to $p$-hyponormality of $\cfw$ we have $0<\hfw<\infty$ a.e. $[\mu_w]$ and so $\efw\big(\hfw^{\alpha}\big)>0$ a.e. $[\mu_w]$ for any $\alpha\in\rbb$. This follows from $\efw\big(\hfw^{\alpha}\big)\circ \phi^{-1}>0$ a.e. $[\mu]$ which in turn can be proved in the following way. Suppose that there exists $\sigma\in\ascr$ such that $\mu(\sigma)>0$ and $\efw\big(\hfw^{\alpha}\big)\circ\phi^{-1}=0$ a.e. $[\mu]$ on $\sigma$. Then we get
\begin{align*}
0=\int_\sigma \efw\big(\hfw^{\alpha}\big)\circ\phi^{-1}\hfw\D\mu=\int_{\phi^{-1}(\sigma)} \efw\big(\hfw^{\alpha}\big)\D\mu_w=\int_{\phi^{-1}(\sigma)} \hfw^{\alpha}\D\mu_w,
\end{align*}
which implies that
\begin{align*}
0=\mu_{w}\big(\phi^{-1}(\sigma)\big)=\int_\sigma\hfw\D\mu.
\end{align*}
This of course contradicts the fact that $\hfw>0$ a.e. $[\mu_w]$.

Now, we fix $q\in(0,\infty)$. According to Proposition \ref{phq} the proof boils down to showing that $\hfw^q\circ\phi\Le \efw\big(\hfw^q\big)$ a.e. $[\mu_w]$. That this inequality holds we deduce from
\begin{align*}
\hfw^q\circ \phi=\big(\hfw^{-p}\circ\phi\big)^{-\frac{q}{p}}\overset{(\dag)}\Le \Big(\efw\big(\hfw^{-p}\big)\Big)^{-\frac{q}{p}}\overset{(\ddag)}\Le \efw\big(\hfw^{q}\big)\quad \text{a.e. $[\mu_w]$},
\end{align*}
where $(\dag)$ follows from Theorem \ref{p-hyp} and $p$-hyponormality of $\cfw$ while $(\ddag)$ follows from
\begin{align*}
1= \efw\Big(\big(\hfw^{-p}\big)^{\frac{q}{p+q}} \big(\hfw^q\big)^{\frac{p}{p+q}}\Big)\Le
\efw^{\frac{q}{p+q}}\big(\hfw^{-p}\big)\cdot\efw^{\frac{p}{p+q}}\big(\hfw^{q}\big)\quad \text{a.e. $[\mu_w]$},
\end{align*}
which is a consequence of the conditional H\"{o}lder inequality (see \cite[Lemma A.1]{b-j-j-sW}).
\end{proof}
Now we employ Theorems \ref{aluthge} and \ref{p-hyp} to arrive at the following.
\begin{thm}\label{ptaszki}
Assume \eqref{stand1}. Let $p\in(0,1)$. If $\cfw$ is $p$-hyponormal and $\hsf_{\phi,w_{1-p}}<\infty$ a.e. $[\mu]$, then $\overline{\Delta_{1-p}(C_{\phi, w})}$ is hyponormal.
\end{thm}
\begin{proof}
By Propositions \ref{phq} and \ref{phq+} we have
\begin{align*}
\hfw^{1-p}\circ\phi\Le \efw\big(\hfw^{1-p}\big)\quad \text{a.e. $[\mu_w]$},
\end{align*}
which, together with \eqref{fifi}, implies
\begin{align*}
\hfw^{1-p}\Le \efw\big(\hfw^{1-p}\big)\circ\phi^{-1}\quad \text{a.e. $[\mu]$}.
\end{align*}
Combining this with monotonicity of $\efw$ and $p$-hyponormality of $\cfw$ we get
\begin{align*}
\efw\Bigg(\frac{\hfw^{p}\circ\phi}{\hfw^p}\frac{\hfw^{1-p}}{\efw\big(\hfw^{1-p}\big)\circ\phi^{-1}}\Bigg)\leqslant \efw\Bigg(\frac{\hfw^{p}\circ\phi}{\hfw^p}\Bigg)\leqslant 1,\quad \text{a.e. $[\mu_w]$}.
\end{align*}
This yields
\begin{align*}
\efw\Bigg(\frac{\efw\big(\hfw^{1-p}\big)\cdot\hfw^{p}\circ\phi}{\efw\big(\hfw^{1-p}\big)\circ\phi^{-1}\cdot \hfw^p}\hfw^{1-p}\Bigg)\leqslant \efw\big(\hfw^{1-p}\big),\quad \text{a.e. $[\mu_w]$},
\end{align*}
which, according to Lemma \ref{relacje}\,(ii), is the same as
\begin{align*}
\efw\Bigg(\frac{\hsf_{\phi, w_{1-p}}\circ\phi}{\hsf_{\phi, w_{1-p}}}\hfw^{1-p}\Bigg)\leqslant \efw\big(\hfw^{1-p}\big),\quad \text{a.e. $[\mu_w]$}.
\end{align*}
Applying Lemma \ref{relacje}\,(iii) we get
\begin{align*}
\esf_{\phi,w_{1-p}}\Bigg(\frac{\hsf_{\phi, w_{1-p}}\circ\phi}{\hsf_{\phi, w_{1-p}}}\Bigg)\cdot \efw\big(\hfw^{1-p}\big)\leqslant \efw\big(\hfw^{1-p}\big),\quad \text{a.e. $[\mu_w]$},
\end{align*}
from which we deduce hyponormality of $\overline{\Delta_{1-p}(\cfw)}$.
\end{proof}
\begin{rem}
Another way of proving Theorem \ref{ptaszki} is to use the inequality
\begin{align}\label{ups}
\efwa\bigg(\frac{\hfwa^{p+\alpha}\circ\phi}{\hfwa^{p+\alpha}}\bigg)\Le\bigg(\frac{\efw(\hfw^\alpha)}{\hfw^\alpha}\bigg)^{p+\alpha-1}\quad \text{a.e. $[\mu_{w_\alpha}]$},
\end{align}
which holds whenever $p\in(0,\infty)$, $\alpha\in (0,1]$, $\cfw$ is $p$-hyponormal and $\hfwa<\infty$ a.e. $[\mu]$. That this inequality is satisfied can be shown in a following way. First, using Propositions \ref{phq} and \ref{phq+} we get
\begin{align*}
\hfw^\alpha\Le \efw(\hfw^\alpha)\circ\phi^{-1}\quad \text{a.e. $[\mu]$},
\end{align*}
which, by Lemma \ref{relacje}\,(ii), implies
\begin{align}\label{sierpien}
\frac{\hfw^\alpha}{\hfwa^{p+\alpha}}\Le \frac{1}{\hfw^p}\quad\text{a.e. $[\mu]$}.
\end{align}
Then, by Lemma \ref{relacje}\,(iii), \eqref{sierpien}, and $p$-hyponormality of $\cfw$, we get
\begin{align*}
\efwa\bigg(\frac{\hfwa^{p+\alpha}\circ\phi}{\hfwa^{p+\alpha}}\bigg)
&=\efw^{p+\alpha-1}(\hfw^\alpha)\,\hfw^{(1-\alpha)(p+\alpha)}\circ\phi \,\efw\bigg(\frac{\hfw^\alpha}{\hfwa^{p+\alpha}}\bigg)\\
&\Le \efw^{p+\alpha-1}(\hfw^\alpha)\,\hfw^{(1-\alpha)(p+\alpha)}\circ\phi\, \efw\Big(\frac{1}{\hfw^p}\Big)\\
&\Le\bigg(\frac{\efw(\hfw^\alpha)}{\hfw^\alpha}\bigg)^{p+\alpha-1}\quad \text{a.e. $[\mu_{w_\alpha}]$},
\end{align*}
which gives the desired inequality.
\end{rem}
Partial hyponormality of the $\alpha$-Aluthge transform of $\cfw$ in the case $\alpha\leqslant1-p$ is covered by the following.
\begin{thm}\label{second}
Assume \eqref{stand1}. Let $p\in(0,1)$ and $\alpha\in(0,1-p]$. If $\cfw$ is $p$-hyponormal and $\hfwa<\infty$ a.e. $[\mu]$, then $\overline{\Delta_{\alpha}(C_{\phi, w})}$ is $(p+\alpha)$-hyponormal.
\end{thm}
\begin{proof}
We first observe that $\hfwa>0$ a.e. $[\mu_{w_{\alpha}}]$. This comes via Lemma \ref{relacje}\,(ii) from the fact that $\hfw>0$ a.e. $[\mu_w]$ implies $\efw\big(\hfw^{\alpha}\big)\circ\phi^{-1}>0$ a.e. $[\mu_w]$ (see the proof of Proposition \ref{phq+}). Next, in view of Propositions \ref{phq} and \ref{phq+}, we see that
\begin{align*}
\Big(\hfw^{\alpha}\circ\phi\Big)^{1-p-\alpha}\leqslant\Big(\efw\big(\hfw^{\alpha}\big)\Big)^{1-p-\alpha}\quad \text{a.e. $[\mu_w]$,}
\end{align*}
which gives
\begin{align}\label{karmeliet1}
\efw^{p+\alpha}\big(\hfw^{\alpha}\big)\Big(\hfw\circ\phi\Big)^{\alpha-\alpha p-\alpha^2}\leqslant\efw\big(\hfw^{\alpha}\big)\quad \text{a.e. $[\mu_w]$.}
\end{align}
As a consequence, by Lemma \ref{relacje}\,(ii), we get
\begin{align*}
\efw\Bigg(\frac{\hfwa^{p+\alpha}\circ\phi}{\hfwa^p}\Bigg)&=\efw\Bigg(\frac{\efw^{p+\alpha}\big(\hfw^\alpha\big)\big(\hfw\circ\phi\big)^{(1-\alpha)(p+\alpha)}}{\efw^{p}\big(\hfw^{\alpha}\big)\circ\phi^{-1}\cdot\hfw^{(1-\alpha)p}}\Bigg)
\\
&=\efw\Bigg(\frac{\efw^{p+\alpha}\big(\hfw^\alpha\big)\big(\hfw\circ\phi\big)^{\alpha -\alpha p-\alpha^2} \hfw^p\circ\phi}{\efw^{p}\big(\hfw^{\alpha}\big)\circ\phi^{-1}\cdot\hfw^{(1-\alpha)p}}\Bigg)\\
&\leqslant
\efw\big(\hfw^{\alpha}\big) \efw\Bigg(\frac{ \hfw^p\circ\phi}{\efw^{p}\big(\hfw^{\alpha}\big)\circ\phi^{-1}\cdot\hfw^{(1-\alpha)p}}\Bigg)
\quad \text{a.e. $[\mu_w]$.}
\end{align*}
Since by Propositions \ref{phq} and \ref{phq+} we have,
\begin{align*}
\efw^p\big(\hfw^\alpha\big)\circ\phi^{-1}\hfw^{(1-\alpha)p}\geqslant \hfw^p,\quad \text{a.e. $[\mu]$},
\end{align*}
we see that
\begin{align*}
\efw\Bigg(\frac{\hfwa^{p+\alpha}\circ\phi}{\hfwa^p}\Bigg)&\leqslant
\efw\big(\hfw^{\alpha}\big) \efw\Bigg(\frac{ \hfw^p\circ\phi}{\hfw^{p}}\Bigg)\leqslant \efw\big(\hfw^{\alpha}\big)
\quad \text{a.e. $[\mu_w]$.}
\end{align*}
This in turn, by Lemma \ref{relacje}\,(iii), yields
\begin{align*}
\efwa\Bigg(\frac{\hfwa^{p+\alpha}\circ\phi}{\hfwa^{p+\alpha}}\Bigg)\leqslant 1 \quad \text{a.e. $[\mu_{w_\alpha}]$.}
\end{align*}
Applying Theorems \ref{aluthge} and \ref{p-hyp} completes the proof.
\end{proof}
Combining Theorems \ref{ptaszki} and \ref{second}, and  Proposition \ref{pq} we get an unbounded weighted composition operator counterpart of the classical result on hyponormality of the Aluthge transform of $p$-hyponormal operators (see \cite[Theorem 1, p.\ 158]{fur}).
\begin{cor}
Assume \eqref{stand1}. Let $p\in(0,\infty)$. Suppose that $\cfw$ is $p$-hyponormal and $\hsf_{\phi,w_{1/2}}<\infty$ a.e. $[\mu]$. Then the following two assertions are satisfied$:$
\begin{itemize}
\item[(i)] $\overline{\Delta_{1/2}(C_{\phi, w})}$ is $(p+\frac12)$-hyponormal for every $p\in\big(0,\frac12\big)$,
\item[(ii)] $\overline{\Delta_{1/2}(C_{\phi, w})}$ is hyponormal for every $p\in\big[\frac12,\infty)$.
\end{itemize}
\end{cor}
Concerning the property of improving $p$-hyponormality by the $\alpha$-Aluthge transformation, it is worth mentioning that applying the transformation with a larger parameter doesn't necessarily mean we obtain an operator with better properties. Below we supply an example of a composition operator $C_\phi$ such that its $\alpha$-Aluthge transform $\Delta_{\alpha}(C_\phi)$ is hyponormal for some $\alpha\in(0,1)$ but any $\beta$-Aluthge transform $\Delta_\beta(C_\phi)$ with $\beta>\alpha$ is not hyponormal.
\begin{exa}
Let $X=\rbb^2$, $\ascr=\borel{\rbb^2}$, and $\mu=\mu^{\exp(\|\cdot\|^2)}$, where $\|\cdot\|$ is the Euclidean norm on $\rbb^2$ (see Example \ref{properinclusion}). Let $\phi\colon\rbb^2\to\rbb^2$ be a linear invertible transformation defined by $\phi(x_1,x_2)=(\theta x_2, x_1)$, $(x_1,x_2)\in\rbb^2$, with $\theta\in(0,\infty)\setminus\{1\}$ to be fixed later. Put $w=\chi_X$.

By \eqref{waga} and \eqref{rn-matrix}, the measures $\mu$ and $\mu_{w_\alpha}$ are mutually absolutely continuous or every $\alpha\in(0,1)$. Hence, in view of Theorem \ref{p-hyp}, for any given $\alpha\in(0,1)$, $\cfwa$ is hyponormal if and only if
\begin{align*}
\text{$\hfwa\circ\phi\leqslant \hfwa$ a.e. $[\mu]$}
\end{align*}
(we use here the fact that $\phi$ is invertible, which implies that $\efwa$ acts as the identity). Using \eqref{rn-matrix}, we easily show that the above is equivalent to
\begin{align}\label{dunkierka}
\exp\bigg((\alpha-1)\|\phi(x)\|^2+(2-3\alpha)\|x\|^2+(3\alpha-1)\|\phi^{-1}(x)\|^2-\alpha\|\phi^{-2}(x)\|^2\bigg)\leqslant 1,\quad x\in\rbb^2.
\end{align}
Since $\phi^{-1}(x_1,x_2)=\big(x_2,\frac{1}{\theta}x_1\big)$ and $\phi^{-2}(x_1,x_2)=\big(\frac{1}{\theta}x_1,\frac{1}{\theta}x_2\big)$ for all $(x_1,x_2)\in\rbb^2$, the inequality \eqref{dunkierka} can be expressed as
\begin{align}\label{stages}
\tag{$\alpha,\theta$}
\begin{cases}
(1-2\alpha)\theta^2+2\alpha-1\leqslant0,\\
\theta^4(\alpha-1)+\theta^2-\alpha\leqslant0.
\end{cases}
\end{align}
Summing up, $\cfwa$ is hyponormal if and only if \eqref{stages} is satisfied.

Analyzing linear mappings $t\mapsto (1-2t)\theta^2+2t-1$ and $t\mapsto \theta^4(t-1)+\theta^2-t$ we may show that for any fixed $\theta\in(0,1)$ the set of all $\alpha$'s such that \eqref{stages} is satisfied is an interval of the form $[r,\frac12]$. Therefore, in order to obtain the desired example, it suffices to consider $C_\phi$ induced by $\phi$ with any $\theta\in(0,1)$, $\alpha=\frac12$, and $\beta>\frac12$.

We point out here that the composition operator $C_\phi$ itself cannot be $r$-hyponormal for any $r\in(0,1)$.
\end{exa}
\section{Quasinormality}
Below we show that a weighted composition operator has to be quasinormal whenever it is a fixed point of the $\alpha$-Aluthge transformation (cf. \cite[Proposition 1.10]{j-k-p-ieot-2000}).
\begin{pro}
Assume \eqref{stand1}. Let $\alpha\in(0,1]$. Suppose that $\cfw$ is densely defined. Then $\cfw$ is quasinormal if and only if $\dfwa=\cfw$.
\end{pro}
\begin{proof}
Suppose $\cfw$ is quasinormal. Then, in view of \cite[Theorem 20]{b-j-j-sW}, $\hfw\circ\phi=\hfw$ a.e. $[\mu_w]$. This implies that $w_\alpha=w$ a.e. $[\mu]$. Applying \cite[Proposition 7]{b-j-j-sW} we get $\hfwa=\hfw$ a.e. $[\mu]$. As a consequence, employing Lemma \ref{relacje}, Theorem \ref{aluthge}, and \eqref{dziedzina}, we get
\begin{align*}
\dz{\dfwa}=L^2\Big(\big(1+\hfw+\hfw^{1-\alpha}\big)\D\mu\Big)\text{ and }\dz{\cfwa}=L^2\Big(\big(1+\hfw\big)\D\mu\Big).
\end{align*}
Since $\hfw^{1-\alpha}\Le1+\hfw$ a.e. $[\mu]$, we deduce from \cite[Lemma 12.3]{b-j-j-s-2014-ampa} that $\dz{\dfwa}=\dz{\cfwa}$. Therefore, by Theorem \ref{aluthge}, $\dfwa=\cfw$.

Suppose that $\dfwa=\cfw$ holds. This, the fact that weighted composition operators are closed, and Theorem \ref{aluthge} imply that $\cfwa=\cfw$. Let $\{X_n\}_{n=1}^\infty\in\ascr$ satisfy $\mu(X_n)<\infty$ and $\hfw<n$ on $X_n$ for every $n\in\nbb$, and $\bigcup_{n=1}^\infty X_n=X$ (such a sequence exists due to $\sigma$-finiteness of $\mu$ and $\hfw<\infty$ a.e. $[\mu]$). Using \eqref{dziedzina} we see that for every $n\in\nbb$, $\chi_{X_n}\in\dz{\cfw}$ and $w\cdot (\chi_{X_n}\circ\phi)=w_\alpha\cdot (\chi_{X_n}\circ\phi)$. Hence, $w=w_\alpha$ a.e. $[\mu]$ on $\phi^{-1}(X_n)$ for every $n\in\nbb$. Since $\bigcup_{n=1}^\infty\phi^{-1}(X_n)=X$, we deduce that $w=w_\alpha$ a.e. $[\mu]$. This yields $\hfw=\hfw\circ\phi$ a.e. $[\mu_w]$. Applying \cite[Theorem 20]{b-j-j-sW} completes the proof.
\end{proof}

\section*{Acknowledgements}
The authors would like to thank Professor Jan Stochel for his support and encouragement during preparation of the paper.

The second-named author was partially supported by the SONATA BIS grant no. UMO-2017/26/E/ST1/00723 financed by the National Science Centre,
Poland.

The research was partially done while the second-named author was visiting CEMPI and University of Lille 1. He wishes to thank the faculty and stuff of both the units for the support and hospitality.
\bibliographystyle{amsalpha}

\end{document}